\documentclass{amsart}

\usepackage{comment, graphicx, color}

\newtheorem{theorem}{Theorem}[section]

\newtheorem{lemma}[theorem]{Lemma}

\theoremstyle{definition}

\newtheorem{definition}[theorem]{Definition}

\newtheorem{question}[theorem]{Question}

\newtheorem{note}[theorem]{Notation}

\newcommand{\RR}{\mathbb{R}}

\newcommand{\CC}{\mathbb{C}}

\newcommand{\T}{\rm{tr\;}}

\newcommand{\QQ}{\mathbb{Q}}

\newcommand{\HH}{\mathbb{H}}

\newcommand{\NN}{\mathbb{N}}

\numberwithin{equation}{section}

\begin{document}

\title[Computability Models]{Computability Models: Algebraic, Topological and Geometric Algorithms}









\begin{abstract} 

The discreteness problem for finitely generated subgroups of $PSL(2,\RR)$ and $PSL(2,\CC)$ is  a long-standing open  problem. In this paper we consider whether or not this problem is decidable by an algorithm. Our main result is that the answer depends upon what model of computation is chosen.  Since our discussion involves the disparate topics of computability theory and group theory, we include substantial background material.

\end{abstract}

\maketitle



\section{Introduction}

\vskip .03in

The discreteness problem for finitely generated subgroups of $PSL(2,\RR)$ and $PSL(2,\CC)$ is  a long-standing open  problem. In this paper we consider whether or not this problem  is decidable by an algorithm. Our main result is that the answer depends upon what model of computation is chosen.  Since our discussion involves the disparate topics of computability theory and group theory, we include substantial background material.

\vskip .03in

We ask, {\sl Can one implement a geometric algorithm on a computer?}  A geometric algorithm is  one given by geometric computations in hyperbolic two-space,  $\HH^2$. The answer is more complicated than one might think.
In this paper we address the issue.


\vskip .03in
Our work is motivated by M. Kapovich's paper  {\sl Discreteness is undecidable.} \cite{Kapovich}.
By this abbreviated statement Kapovich means that the set of finitely generated discrete subgroups of  $PSL(2,\CC)$ is not computable in the sense of Blum-Shub-Smale \cite{BSS}.

\vskip .03in

A Blum-Shub-Smale machine, a BSS machine for short,  is one model of computation.



%



\vskip .03in

Here we discuss results in different computational models. These include
{\it bit computability} 
and other models discussed,  along with results about their complexity, in   \cite{GM, GMem, GilmanComplexity, Gshort}.



\vskip .03in

We consider  the various models through the lens of the  two-generator $PSL(2, \RR)$-discreteness question: Namely,
\vskip .03in
\begin{question} \label{question:one} Given two elements $A$ and  $B$ in $PSL(2,\RR)$ is the group they generate, $G = \langle A,B \rangle$, discrete and non-elementary?
\end{question}
\vskip .03in

In  Theorem \ref{theorem:summary}   we show that even though the  Gilman-Maskit \cite{GM} discreteness algorithm establishes decidability  of $PSL(2,\RR)$-discreteness in the geometric model, the  BSS model and the symbolic computation model, the algorithm does not work in the  bit-computation model and not every step of the GM algorithm is computable in this model.

\medskip

This yields our main theorem:

\begin{theorem} \label{theorem:notbothplus}
A problem (resp. a set)  which is decidable (resp. computable)  in one model of computation is not necessarily decidable (resp. computable)   translated to another model of computation.
\end{theorem}



The figures above  illustrate some configurations that the Gilman-Maskit  geometric algorithm might visit as it passes through the steps.

\begin{figure} \label{fig:A}\begin{center}
\includegraphics[width=3.1cm,keepaspectratio=true]{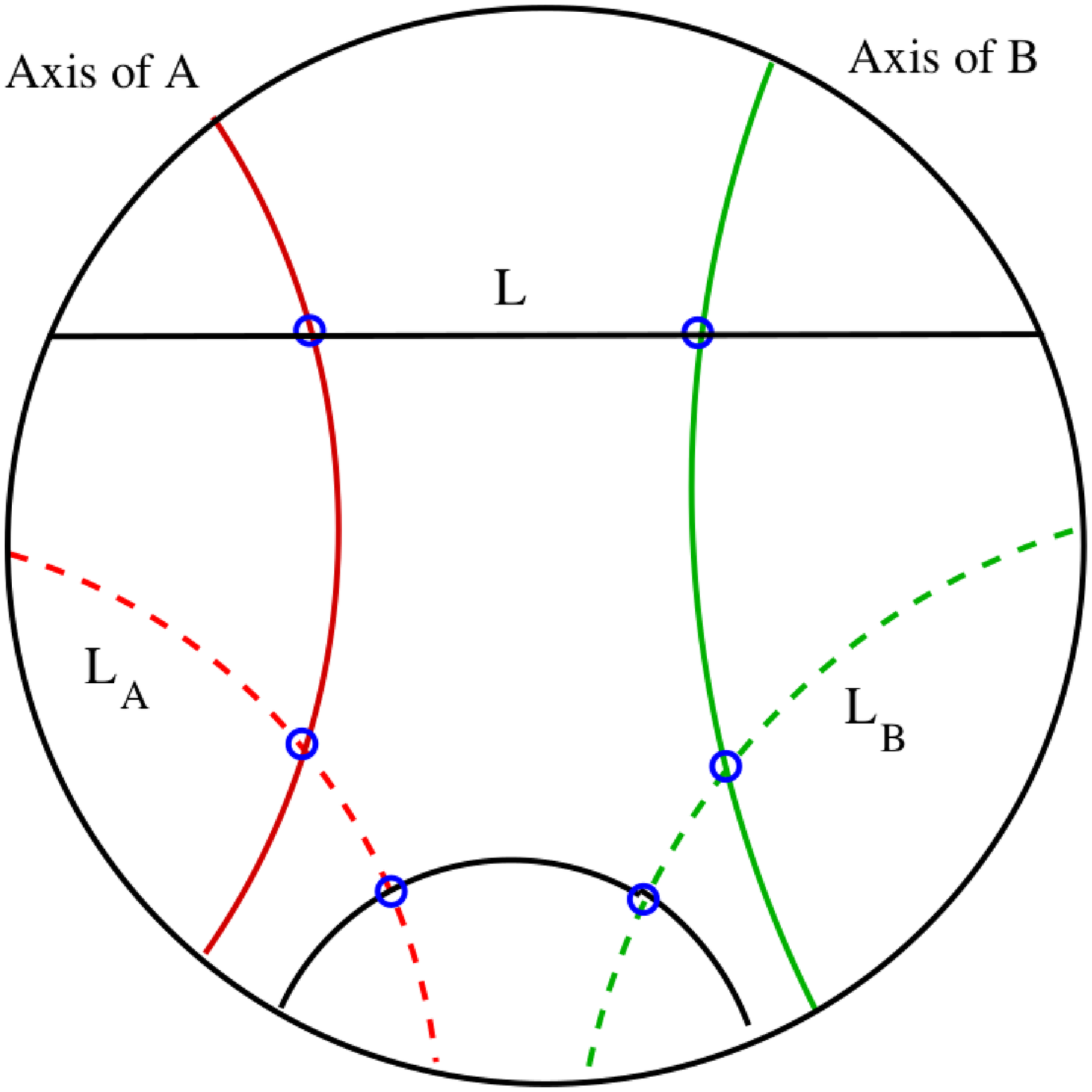}
$\;\;\;$\includegraphics[width=3.0cm,keepaspectratio=true]{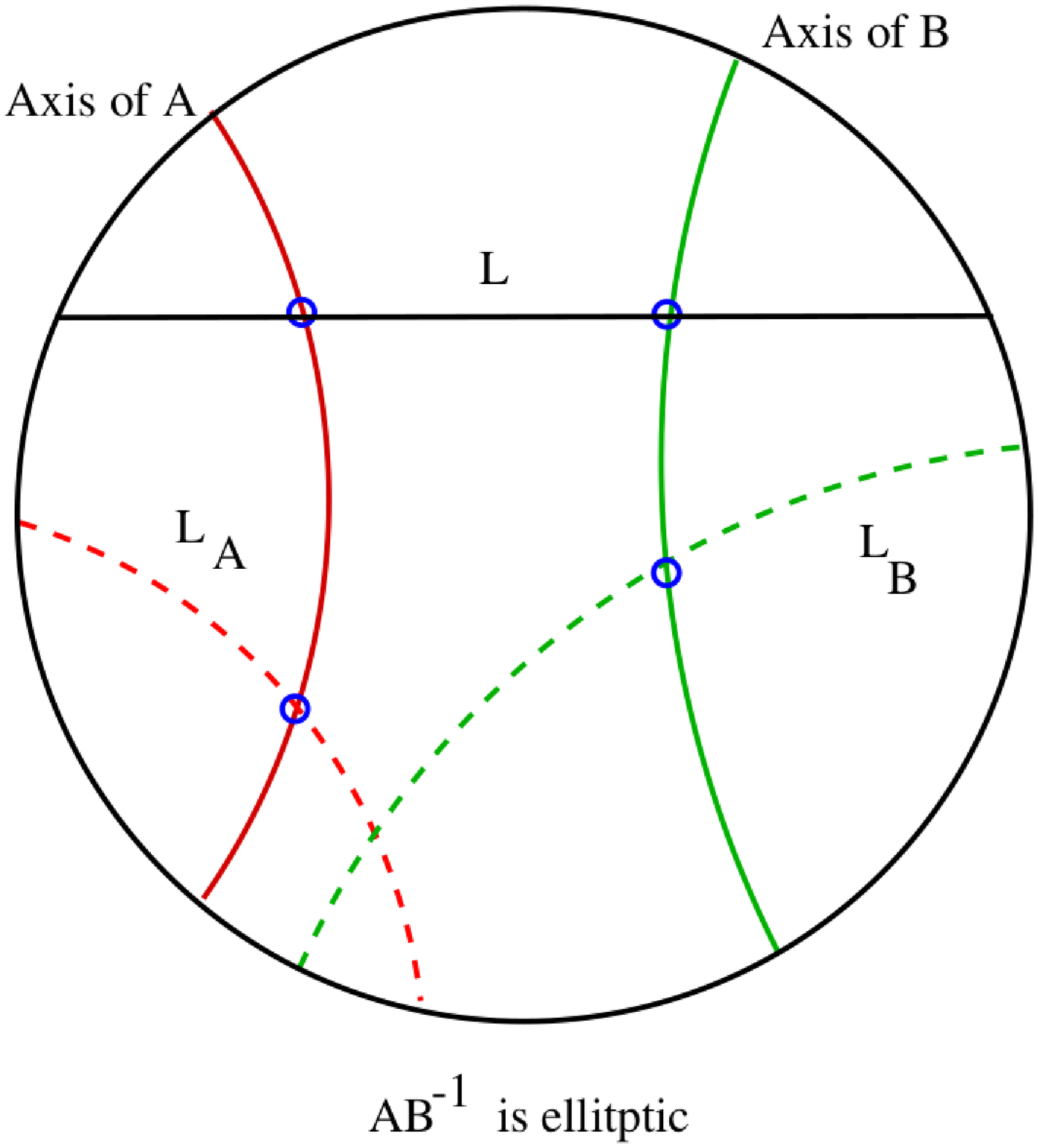}
\includegraphics[width=3.2cm,keepaspectratio=true]{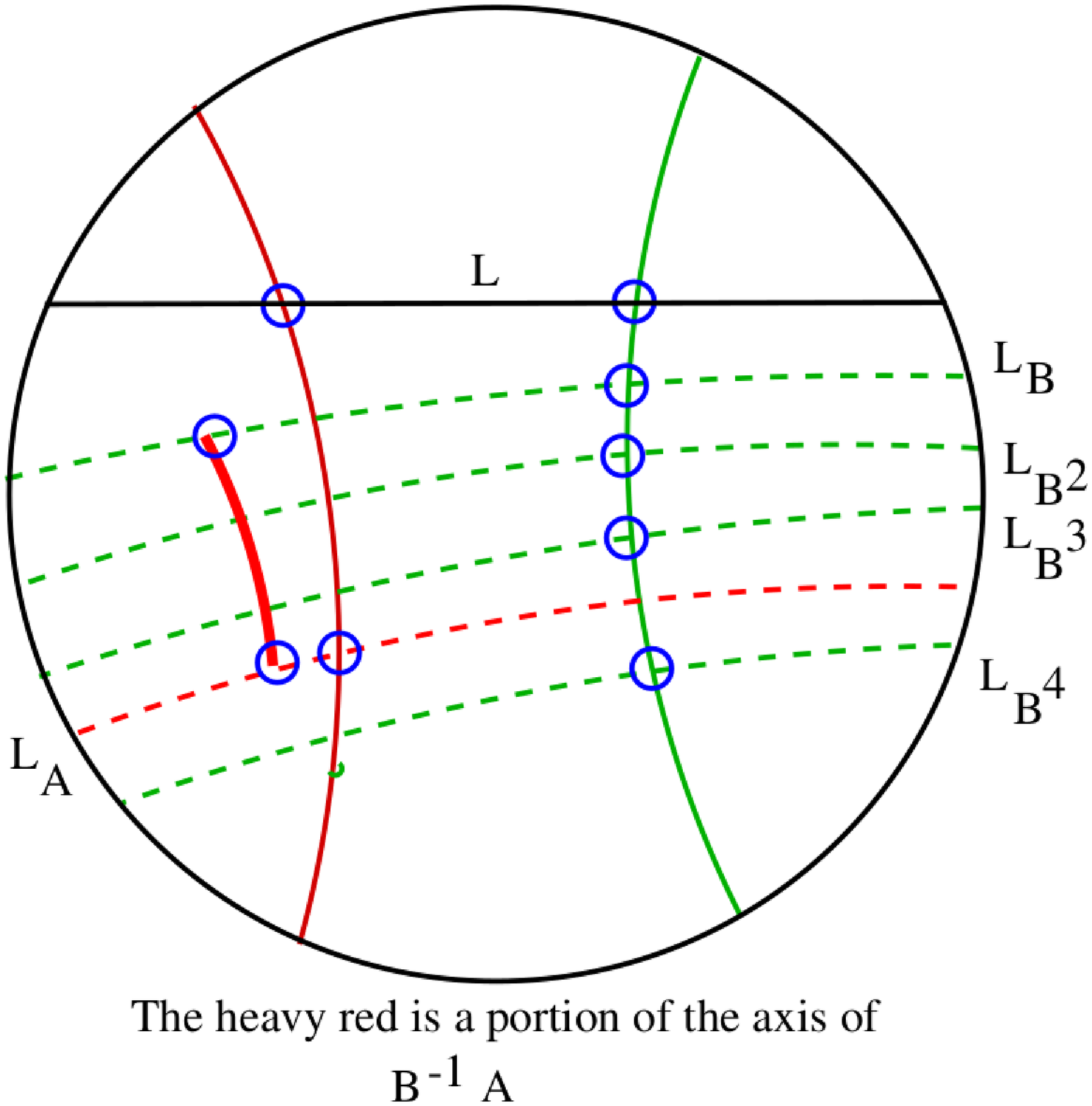}
\end{center}
\caption{Axes and Half-turn Axes Configurations}
\end{figure}

 \medskip

As a specific example, we consider $SL(2, \RR)$-decidability and we show that for  Riemann surfaces of type $(g,n)$, of genus $g$ with $n$ punctures, $T(g,n)$,   the Teichm{\"u}ller  Space,
 and $R(g,n)$ (the  rough fundamental domain for the action of the Teichm{\"u}ller modular group on $T(g,n)$ in the sense  of \cite{keen})  are BSS decidable (Theorem \ref{theorem:Tgn}).




\medskip

The organization of this paper is as follows: Section \ref{sec:background} contains background of both historical and technical nature and preliminary material. It presents  a heuristic definition of an algorithm,  some details of the Gilman-Maskit geometric algorithm and an example of the relation of algebra and geometry. Section \ref{sec:candd} reviews the definitions of computability and decidability.
Section \ref{sec:compmodels} defines what is meant by computation models and these models are discussed in sections
\ref{sec:geomalg}, \ref{sec:BSS}, \ref{sec:symboliccompalgebra}, \ref{sec:bitcomputability}, touching respectively upon geometry, BSS, symbolic computation and bit-computability. For the most part Sections \ref{sec:BSS} and \ref{sec:symboliccompalgebra} contain definitions and summaries of earlier results. The latter are given to place the main theorem in context. The BSS-decidability of $T(g,n)$,  Theorem \ref{theorem:Tgn},   is proved in section \ref{sec:Tspace}. Section \ref{sec:conclusion} summarizes the results.
The last section, Section \ref{sec:questions},  lists some open questions.

\section{Background and Notation}\label{sec:background}

\vskip .03in
Question \ref{question:one} is an old question, one  that mathematicians thought should be easy to solve. As a result there were a number of papers published that had omissions and errors. It turned out that the solution required an algorithm.
\vskip .03in
\subsection{What is an algorithm?} 

 \vskip .03in

  Heuristically one can regard an algorithm as a recipe for solving a problem, a recipe that is composed of allowed simple steps and a recipe that always gives the right answer. Since a recipe that does not stop does not give the right answer, the definition implies that an algorithm always comes to a stopping point.\footnote{Riley's \cite{Riley}
work produces a {\sl procedure} and not an algorithm because it will not necessarily stop}
However, this definition leaves a lot of room for deciding what the allowed simple steps are. Here we distinguished
models of computation in part by consideration of what allowed simple steps for the GM algorithm are included.

\subsection{The Gilman-Maskit algorithm}

\vskip .05in

There is a geometric algorithmic solution to the $PSL(2,\RR)$ discreteness problem due Gilman and Maskit \cite{GM} .

The Gilman-Maskit algorithm, the GM algorithm for short, consists of
easy geometric and computational steps and stops in finite time with an explicit  bound.

Pull back to $(\tilde{A},\tilde{B}) \in SL(2,\RR)$ with $\T >0$, $\det =1$ and
let $T$ be the maximal initial trace. That is

$T = max\{ |\T \; \tilde{A}|, |\T \; \tilde{B}|, |\T \; \tilde{A}\tilde{B}|, |\T \; \tilde{A}\tilde{B}^{-1}|\}$

The implementation of the algorithm
 dovetails tests for J{\o)rgensen 's inequality and hypotheses of the  Poincare Polygon Theorem.

\begin{note} Often in what follows
for ease of exposition we will not distinguish notationally between
$X \in PSL(2,\RR)$ and its pull back to $SL(2,\RR)$.
unless it is not clear where the matrix we are referring to sits. We often fail to distinguish between a matrix and its action as a M{\"o}bius transformation acting on the hyperbolic plane, $\HH^2$.
\end{note}
\medskip
The algorithm proceeds by considering a given pair of M{\"o}bius transformations. There is an order placed on the transformations with hyperbolics harder than parabolics which are in turn harder than elliptics. This induces an order on pairs of transformations. The algorithm considers a pair and either decides that it is not discrete, usually by using J{\o}rgensen's inequality,  or determines that is it discrete, usually by using the Poincare Polygon Theorem. If neither discreteness or non-discreteness is determined, the algorithm produces a next pair to consider.  The next pair of generators comes from the previous pair using a Nielsen transformation that only changes one of the two generators. The next pair may be in an easier case or it may be in the same type of case but the algorithm assures that the same type of case is repeated at most a finite number of times. Thus the algorithm stops and produces an answer after a finite number of pairs is considered.

\subsection{Why emphasize two-generators?}

$\;$
\vskip .05in

A consequence of J{\o}rgensen's inequality is
 \begin{theorem} {\rm (J{\o}rgensen, \cite{Jorg})}
 Let $G$ be a finitely generated subgroup of $PSL(2,\RR)$. $G$ is discrete if and only if   $ \langle A, B \rangle$ is discrete for every pair of $(A,B) \in G$ which generate a non-elementary group.
 \end{theorem}

 This makes two-generator groups especially important.

\subsection{Example: Geometric and Algebraic Equivalence } \label{sec:Gequiv}

\vskip .05in

We are interested in translating a problem from one model of computation to another.
We give an example of a theorem that translates a geometric condition into an algebra condition and vice-versa.
We first note that elements of $ PSL(2,\mathbb{C})$ are classified algebraically by  their traces and  equivalently by the corresponding action on $\HH^3$, hyperbolic three-space, geometric conditions.
We remind the reader of the definition of coherent orientation.
If $A$ and $B$ are hyperbolic transformations with pull backs $\tilde{A}$ and $\tilde{B}$ to $SL(2,\RR)$
${\T {\tilde{A}}} \ge \T {\tilde{B}} > 2$ with  $L$ the common perpendicular to their axes, then $A$ and $B$ are coherently oriented if the attracting fixed points of ${\tilde{A}}$ and ${\tilde{B}}$ lie to the left of $L$ when $L$ is oriented from the axis of ${\tilde{A}}$ towards the axis of ${
\tilde{B}}$ (see Figure 2).\footnote{
Let $R_X$ denote reflection in the hyperbolic geodesic $X$. Given $A$ and $B$ there are hyperbolic
geodesics $L_A$ and $L_B$ such that  $A= R_{L_A}R_L$ and $B= R_{L_B}R_L$ and $AB^{-1} = R_{L_A}R_{L_B}$. The algorithm proceeds by considering the possible configurations for $L$,  $L_A$ and $L_B$.
The reflection geodesics are pictured in some of the figures but are not needed here.}

\begin{figure} \label{fig:B}  \begin{center}

\includegraphics[width=3cm,keepaspectratio=true]{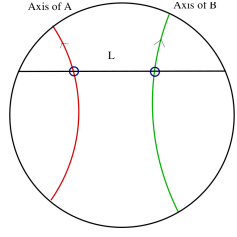}

\caption{Hyperbolics with Coherently Oriented Axes}
\end{center}

\end{figure}
 We take the trace of an element $X
\in PSL(2,\RR)$, $\T X$,   to be the trace of the appropriate pull back.

\medskip

\noindent \begin{theorem} {\rm (Gilman-Maskit \cite{GM} (1997) page 17)}

\vskip .04in

\centerline{\rm The Geometric Meaning of Negative Trace}  

\vskip .05in

\noindent Assume that $\T A \ge \T B > 2$ and that the axes of $A$ and $B$ are coherently oriented so that $\T AB > 2$. Then

\begin{enumerate}
\item $\T AB^{-1} < -2$ $\iff$ the axes of $A$, $B$ and $AB^{-1}$ bound a region,  and
\item $\T AB^{-1} >  \, 2 $ $\iff$  one of the axes of $A$, $B$ and $AB^{-1}$ separates the other two.
 \end{enumerate}

 \end{theorem}

 This is, of course, one of the two main theorems proved in establishing the GM algorithm and showing that it is a true algorithm. The other result is that the procedure replacing  a pair of generators by the appropriate  Nielsen equivalent  pair stops because in the presence of J{\o}rgensen's inequality, there is a positive lower bound by which traces decrease under a Nielsen transformation.
\begin{figure}\label{fig:C} \begin{center}

\includegraphics[width=3.0cm,keepaspectratio=true]{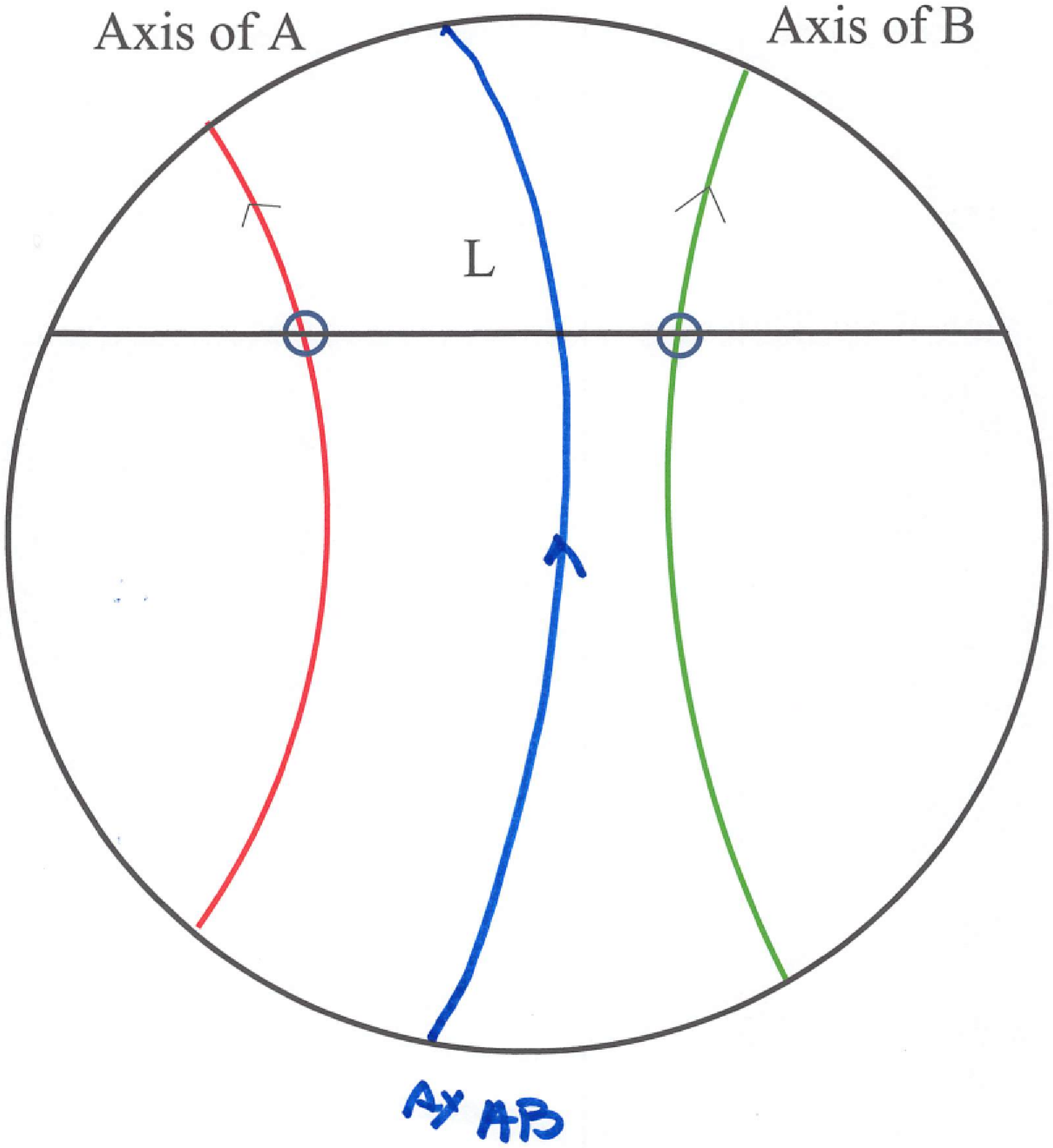}

\includegraphics[width=3.0cm,keepaspectratio=true]{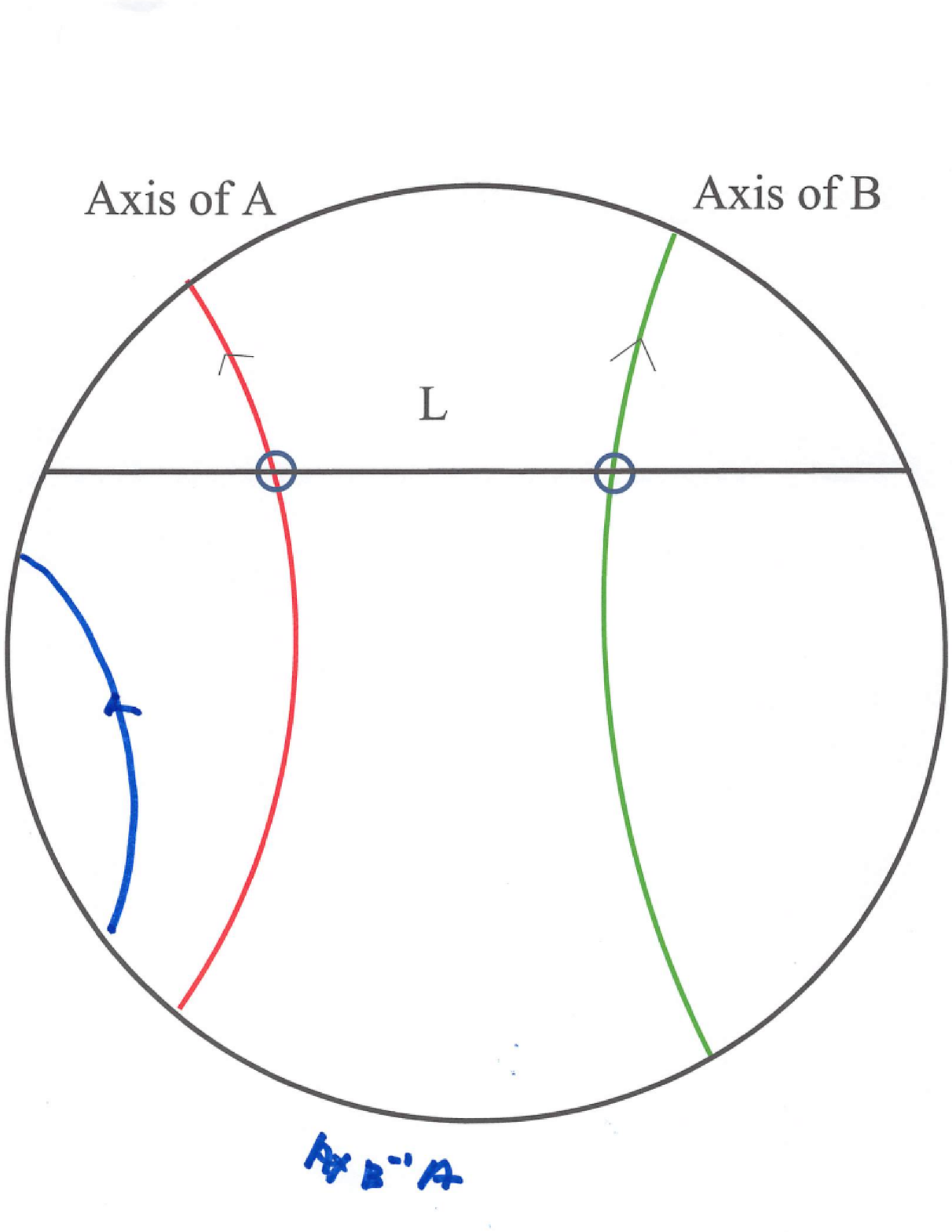} $\;\;$
\includegraphics[width=3.0cm,keepaspectratio=true]{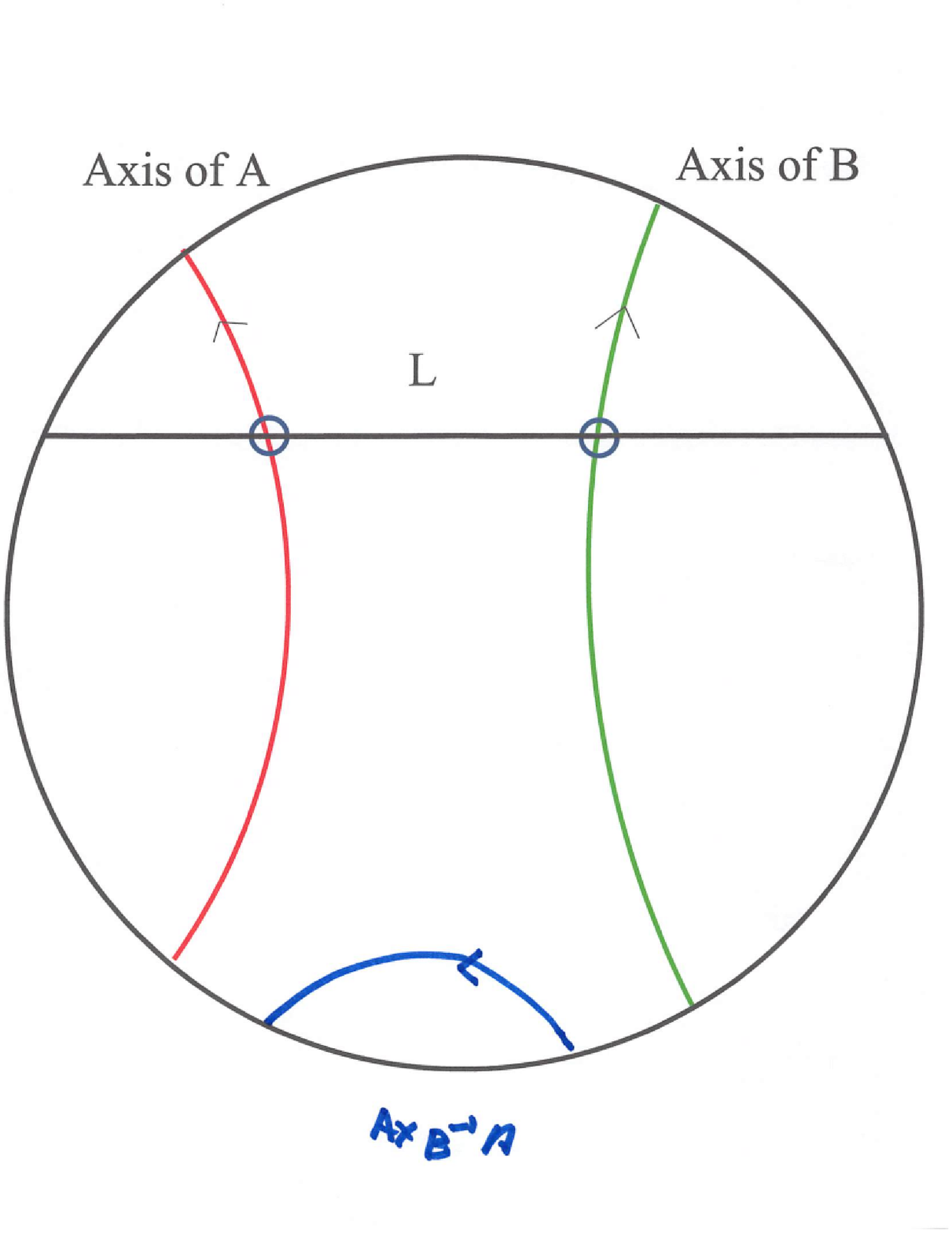}

 \end{center}
\caption{Axis $BA$ separates and Axis $B^{-1}A$  bounds}\end{figure}

\subsection{Overview of the GM Algorithm \label{sec:overviewAlgorithm}}

The algorithm begins with a pair of elements  $(A,B)$ in $PSL(2,\RR)$ and works with pull backs to $SL(2,\RR)$ as needed.

\medskip

The algorithm assumes that beginning pair are  a coherently oriented pair of hyperbolics and involves a number of steps to process the pair before it considers a {\it next} pair of generators or stops. Specifically:
\begin{enumerate}\vskip .1in
\item \label{item:1} If $\T \; AB^{-1} \le -2$,then  $G$ is discrete and free.  
\item \label{item:2} If $\T \; AB^{-1} > 2$, then  repeat \ref{item:1} and \ref{item:2} with the coherently oriented pair  $(AB^{-1},B)$ or $(B, AB^{-1}$)
\item If $\T \;  AB^{-1}  =  \pm 2$, then go  parabolic cases.
\item If $ -2  < \T \;  AB^{-1} < 2$, then  $G$ is either not free or not discrete.

\end{enumerate}

For ease of exposition we omit the details of cases that involve parabolics or elliptics. The algorithm can start with a pair that involves parabolics or elliptics. The complexity of all cases is analysed in \cite{GilmanComplexity} and repeated in the results of theorems \ref{theorem:partI}
and \ref{theorem:partII}. The complexity of the hyperbolic-hyperbolic case is higher than other cases and thus dominates.

\section{Computability and  Decidability} \label{sec:candd}
The precise definition
of a computable function for an oracle Turning machine is given in section \ref{sec:bitcomputability}.

As a starting point  we use Braverman's heuristic definition

\begin{definition} {\rm (\cite{Braverman} definition 1.11)} A function on  a set $S$ is {\it computable} if
there is a Turing machine that takes $x$ as an input and outputs the value $f(x)$.
\end{definition}
However, we modify this definition as we vary the model of computation.
We will be concerned with functions that are computable in the geometric model, in the BSS model, in the symbolic computation model and in the bit-computability model.
\begin{definition}
A Set $S$ is decidable if and only if there is a computable function with $f_S$ with
$$f_S(x) =           \left\{
                  \begin{array}{ll}
                    $1$, & \hbox{$ \iff x \in S $;} \\
                    $0$, & \hbox{$\iff x \notin S $.}
                  \end{array}
                \right.$$

\end{definition}
Sets are decidable (or undecidable) whereas functions are computable (or not).

\begin{definition}
An algorithm is decidable if its stopping set is decidable.
\end{definition}

We cite Braverman for a definition of a {\sl Turing Machine}. He says page $1$ and $2$ of \cite{Bravermanbook}
"A precise definition of a Turing Machine,TM, is somewhat technical and can be found in all texts on computability
e.g . \cite{Pap,Si}." {\footnote{He continues,   "Such a machine consists of a {\sl tape} and a {\sl head} which can be read/erase
/write the symbols on the tape one at a time and can shift its position on the tape in either direction. The symbols on the tape come from a finite alphabet and the TM can be in one of finitely many {\sl states}. Finally a simple look-up table tells TM which action to undertake depending upon the current state and the symbol read on the tape.}{\footnote{We note that there is inconsistent use of terminology-inconsistent use over time
partially due to recent developments via computer scientist and bit-computability.
Thus we note that what was termed the Real Number Algorithm in [8, 7] would
now simple be termed the BSS machine; what was termed the Turning Machine
(TM) algorithm there would now be termed the symbolic computation algorithm.}

An algorithm will consist of the composition of a number of functions or steps together with  branching tests that involve functions. For an algorithm to be computable, the functions that determine steps and branching are required to be computable.

\section{Computational models: type of  allowed simple steps} \label{sec:compmodels}

We distinguish computation models by the allowed simple steps. 

Heuristically a {\sl Turing Machine algorithm} is one where the simple steps
can be carried out by a computer.
To implement a Turning machine algorithm, a TM algorithm, the input must be finite. Since it may require an infinite amount of information to specify a real number from the set of all real numbers, no TM algorithm can deal with the set of all
real numbers. Thus in addition to BSS machines we consider bit-computability and oracle Turing machines (Section \ref{sec:bitcomputability}) that input computable real numbers and symbolic computation \ref{sec:symboliccompalgebra}, where the input consists of rational polynomials or equivalently their coefficients.

We have described the GM algorithm which is a {\it geometric algorithm}.

Further details of each of the other models  can be found in Sections \ref{sec:geomalg}, \ref{sec:BSS}, \ref{sec:symboliccompalgebra}, and \ref{sec:bitcomputability}.



\section{Geometric Algorithms} \label{sec:geomalg}

A geometric algorithm in hyperbolic two-space is one where the allowed simple steps are geometric in nature. This includes, for example,  determining when two geodesics are disjoint or intersect and when given two disjoint geodesics,  determining when a third disjoint geodesic separates the other two.

To a mathematician working in hyperbolic geometry, the GM algorithm is a stand alone theorem that determines when a non-elementary  two generator group is and is not discrete. The algorithm  needs to be translated to other models because the question is often raised as to whether the algorithm has been implemented. It can be implemented using symbolic computation \cite{GilmanComplexity} but one also needs to think about whether it is bit-computable. This is addressed in Section \ref{sec:bitcomputability}.

\section{Blum-Shub-Smale Machines, BSS Machines} \label{sec:BSS}

We note that the geometric GM algorithm was shown to be decidable in the BSS model \cite{GilmanComplexity}.

We review the details of BSS computation. The reader can also see Braverman \cite{Braverman}.
\vskip .05in

 A BSS machine is one where the simple steps

(i)  include all ordinary arithmetic operations on
real numbers including  comparing the size of two real numbers, finding if two real numbers are equal, and
 determining the sign of a real number.

(ii) include oracles that allow you to compute anything,   e.g. computing the $\arccos$ of a number,     determining the rationality of a real number.

We can consider the steps of a BSS machine as an outline or abstract of the (geometric) algorithm. That is, the sequence of computational moves to be made by the geometric algorithm without taking into account how the input is given or whether
the required computational operations can be implemented on a computer.

In particular our BSS machine has oracles
 so  it {\sl can} decide if  $\arccos x$ is a rational multiple of $\pi$.

Given a BSS machine, the halting is the set $S$ is a set for which given $x \in \RR^n$ and an input vector $x$ the BSS machine stops if and only if $x \in \RR^n$.
If a set $S$ is the halting set of a BSS machine,  it is {\it semicomputable} or {\it semidecideable}. If its complement is also semidecideable, then the set is BSS computable or BSS decidable.

\begin{definition}  A {\it semi-algebraic formula} $\phi(x_1,...,x_n)$ is a finite combination of polynomial
equalities and inequalities over $\RR^n$  linked by the logical connectives of and, or, not.
A {\it semialgebraic set} in $\RR^n$ is the set of points satisfying a semi-algebraic formula.
\end{definition}

\noindent Any semialgebraic set is computable by a BSS machine.

\medskip

We have the following (\cite{BCSS98} Theorem $1$ and \cite{Braverman}).

\begin{theorem} {\rm  (Blum, Shub, Smale \cite{BCSS98}; Braverman \cite{Braverman})}
If a set $C \subset \RR^n$ is decided by a BSS machine,  then $C$ is a countable
union of semi-algebraic sets. \end{theorem}

\subsection{Complexities: BSS machine} \label{sec:complBSS}

We define the {\it algebraic complexity}  of an algorithm to be the number of steps it takes to process an input of given size.
Let $T$ be {\bf the maximal initial trace}, that is $\max\{ |\T  \; A|, |\T \; B|, \T \;  AB|,  |\T \; AB^{-1}| \}$  and $d$ the order of the {\it first finite order elliptic} element the algorithm encounters
\vskip .2in
We recall
\begin{theorem} \label{theorem:partI} {\rm (Gilman \cite{GilmanComplexity}) } {\bf {\rm   Complexity Part I}}
\begin{enumerate}
\item
The {\sl algebraic } complexity of the BSS $PSL(2,\QQ)$  algorithm is $O(T)$.
 \item The {\sl algebraic} complexity of the BSS $PSL(2,\RR)$  algorithm is $O(T^2 +d)$
\item The {\sl algebraic} complexity of one BSS calculation of J{\o}rgensen's inequality is $O(1)$.
\end{enumerate}

\end{theorem}



\section{ Symbolic Computation-Computer Algebra} \label{sec:symboliccompalgebra}
\vskip .05in

We note that discreteness was shown to be decidable in the symbolic computation model \cite{GilmanComplexity}

We assume that the matrix entries lie in a finite simple extension of the rationals $\QQ(\gamma)$
of degree $D$ and that
the  input is given by

$M_{\gamma}$, minimal polynomial for $\gamma$; $D$, the degree;
with the eight representing polynomials
$R_{\alpha_i}$  each with an isolating interval  $Isol_{{\alpha}_i}$ ,  $i = 1, ...8$.

We
measure size of a polynomial by its degree and its seminorm:

If $P(x) = p_t\cdot x^t + ... + p_l\cdot x + p_0$, a polynomial with rational coefficients, $p_i = {\frac{r_i}{s_i}}$ written in lowest terms.

The {\bf seminorm of $P$}, $SN(P) = |r_0| +  |r_1| + \cdots + |r_t| + |s_0| + \cdots + |s_t|$.

We define
$L(SN) = \lfloor \ln(SN) \rfloor + 1$.

$L(SN)$   so behaves computationally very much like a logarithm.

\subsection{$\exists$ Symbolic Computation Algorithms } \label{sec:symbolcomp}

Given $\alpha$, $\beta$, algebraic numbers,
one can do algebraic number arithmetic.

E.g. There are  programs \cite{Collins, Collins2} that input
  $$\{ R_{\alpha}, R_{\beta}, Isol_{\alpha}, Isol_{\beta} \}$$
and output
   $$\{ R_{\alpha + \beta}, \;  Isol_{\alpha + \beta} \}$$
and also estimate the costs so that for example
the cost of addition is $O(D(L(S))^2)$ and that of multiplication is $O(D^3(L(S))^2)$

It is easily seen that the size increases can be estimated by  $SN(AB) \le SN(A)SN(B)$ although better bounds may exist.

\subsection{Contrast: GM BSS and Symbolic Computation Translations}

We contrast the GM algorithm when translated and implemented as a BSS machine with the algorithm translated and  implemented using symbolic computation by comparing the complexities\footnote{Omitting the intersecting axes case and simplifying notation used in \cite{GilmanComplexity}}.

{\begin{theorem}
{\rm (Gilman
  \cite{GilmanComplexity})}  {\bf Complexity Part II} {\rm   Theorem  \ref{theorem:partI}
 continued}  \label{theorem:partII}
\vskip .05in

(4) The {\sl Complexity of $PSL(2,\RR)$ disjoint axes algorithm is}
$$O(2^{k(S_0D)^2})$$
where $k$ is a constant.

$S_0$ replaces $T$, $D$ replaces $d$ where $D$ s the degree of the extension and $S_0$ is the maximum of the semi-norms for the $8$ representing polynomials and the minimal polynomial of $\gamma$

\vskip .03in

(5) The {\sl The Complexity of the {\rm first} implementation of J{\o}rgensen} is : $$O(D^{8}(L(S_0)^2)$$

But at later step we are looking at $S_M$ instead of $S_0$ and $L(S_M)$ again gives us $$O(2^{kS_0^2D^2})$$.

\end{theorem}

\subsection{Exponential growth}

For the GM algorithm, the semi-norm,  $SN$, grows exponentially. We note that the algorithm replaces $(A,B)$ by a new pair and the we say that the new pair is either given by a linear step or a Fibonacci step. Specifically replace $(A,B)$ by $$\left\{
                                        \begin{array}{ll}
                                          (AB^{-1},B), & \hbox{linear step ;} \\
                                          (B,AB^{-1}) , & \hbox{Fibonacci step.}
                                        \end{array}
                                      \right.$$

Note  that Fibonacci steps cause the length of the words subsequently  considered to grow exponentially and that  matrix entries grow exponentially under product. Thus the algorithm is potentially double exponential.

On the other hand, one
can handle elliptics
\begin{lemma}   (Gilman  \cite{GilmanComplexity})

 If $E$ is elliptic of finite order, then its order is bound by $32D^2$
\end{lemma}
Y.C. Jiang revised parts of the algorithm. He used the revised version to show:
\begin{theorem} {\rm (Y. C.  Jiang \cite{Jiang})}

The  symbolic computation algorithm is of polynomial complexity.
\end{theorem}

This is all well and good except for the fact that as hyperbolic geometers we don't think of all of our points as being roots of polynomials.

\section{Bit Computability over the Reals} \label{sec:bitcomputability}

This is an old  topic (see \cite{Ko, Wei, ZB} and reference given there) which has most recently been further developed by Braverman and others
\cite{Braverman, Bravermanbook, BravermanAMS} since $2005$.

We are interested in algorithms over the computable reals which require bit computability. The concept of  bit computability  can be applied to functions and to graphs and to sets. The model of computation used in these cases is an  {\it oracle Turning machine}.

We begin with definitions(see \cite{Bravermanbook}).

\begin{definition}
A function $f(x)$ is computable if there is a TM which takes $x$ as input and outputs the value of $f(x)$.
\end{definition}
\begin{definition}
A real number $\alpha$ is said to be computable if there is a computable function $\phi: \NN \rightarrow \QQ$ such that, for all $n$, $$|\alpha - {\frac{\phi(n)}{2^n}} | < 2^{-n}.$$
\end{definition}
We work with the set of computable real numbers.
\begin{definition} A dyadic-valued function $\phi: \NN \rightarrow \mathbb{D}$ is called an oracle for a real number $x$ if it satisfies
$|\phi(m) -x| < 2^{-m}$ for all $m$.
  \end{definition}

  An {\it oracle Turing machine} is a TM that can query the value $\phi(m)$ of some oracle for an arbitrary $m \in \NN$.
We note that in these cases
\begin{itemize}
\item
Our algorithms cannot say READ $x$.
\item
They say instead

READ $x$ WITH PRECISION $(1/2)^n$.

\item

The TM comes with this command.
\end{itemize}


The oracle terminology separates the problem of computing the parameter $x$ from the problem of computing the function $f$ on a given $x$ and note

\begin{definition}
Let $f:S \rightarrow \RR$. Then {\it $f(x)$ is computable} if is there is an oracle TM that computes $f(x)$ to any desired degree if accuracy. That is given $n$, the machines returns a dyadic $q$ with $|q-x| < (1/2)^n$.
\end{definition}
\begin{definition}
A set {\it $S$ is computable} if there is a function that outputs yes if a point is the
set and no if the point is not in the set.
\end{definition}


\begin{theorem} {\rm (Braverman, Yampolsky \cite{Bravermanbook})}

Let $S \subset {\RR}^k$. Assume that $f: S \rightarrow  \RR^k$ is computable by an oracle TM. Then $f$ is continuous.
\end{theorem}
Thus step  functions are not computable.\footnote{Braverman \cite{BravermanAMS} has more recent work on how to handle so what he terms {\sl simple  functions} and these include step functions. However, these results do not seem relevant to our problem at hand.} 
That is, $x =0$,   $x=2$,  $>$,  $<$ are not computable in this model.
The sign function 
$$sign(x) = \left\{
              \begin{array}{ll}
                $1$ & \hbox{if $x \ge 0$;} \\
                $0$, & \hbox{if $x <0$.}
              \end{array}
            \right.$$
is not computable.
\begin{theorem} \label{theorem:not}
The GM algorithm cannot be translated to a bit computable algorithm and thus it is not decidable in the bit model.
\end{theorem}
In \cite{GT} we address this issue as it arises the attempt translate the GM algorithm to a bit-computable one.  We use ideas that include what we term {\sl extended bit-computable domain dependent algorithms} and multi-oracle Turing Machines that do not use a sign oracle. These notions seem appropriate for the algorithms we want to implement as well as for other unrelated algorithms.

\section{$SL(2,\RR)$-decidability: The Teichm{\"u}ller Space is BSS decidable} \label{sec:Tspace}

 The Riemann space or Moduli space is the quotient of the Teichm{\u"}ller modular group acting on the Teichm{\"u}ller space of a surface of finite type. A surface of finite type, type $(g,n)$,  is a surface of genus $g$ with $n$ punctures where $g$ and $n$ are integers.
It is standard to describe $T(g,n)$, the Teichm{\"u}ller space of a group representing a surface $S$ of finite type by real parameters, trace parameters.  The Riemann space can be thought of as the space whose points are  conformal equivalence classes surface. The Teichm{\u"}ller modular  group acts on the Teichm{\"u}ller space yielding Riemann space as its quotient. 
The space, $R(g,n)$, the  rough fundamental domain of \cite{keen},   is an expanded moduli space and  is a fundamental domain for the action of the Mapping-class group,  but it may contain a finite number of points equivalent to a given point in the Teichm{\"u}ller space.




In theorem 6.1 of \cite{keen} the Teichm{\"u}ller space is given by a series of equalities and inequalities on the trace parameters as is the rough fundamental domain for the action of the Teichm{\u"}ller modular group.



It follows that

\begin{theorem} \label{theorem:Tgn} Let $S(g,n)$ be a surface of type $(g,n)$, of genus $g$ with $n$ punctures.  The Teichm{\"u}ller space of $S(g,n)$, $T(g,n)$,   and $R(g,n)$, the   rough fundamental domain for the action of the Teichm{\u"}ller modular group on $T(g,n)$  are BSS-decidable and are computable sets.
\end{theorem}
%

\begin{proof}
The trace parameters are real parameters for the Teichm{\"u}ller and Riemann Spaces. The inequalities and equalities in the trace parameters of
\cite{keen} Theorem 6.1 exhibit $T(S)$ and the rough $R(S)$ as
semialgebraic sets. By \cite{BSS}  semialgebraic sets are BSS computable.
\end{proof}

\section{Conclusion} \label{sec:conclusion}
We have shown

\begin{theorem} \label{theorem:summary}
The GM algorithm is decidable in the geometric model, the BSS model and the symbolic computation model, it is not decidable in the bit-computable model.  Not every step of the algorithm involves a function that is computable in this model.
\end{theorem}

\begin{proof}
The GM algorithms requires a step that uses a sign function.
\end{proof}

This proves our main theorem, Theorem \ref{theorem:notbothplus}.


\section{The Overriding Question} \label{sec:questions}

\medskip
\vskip .05in
{\centerline{\bf {\it What is the appropriate computational model?}}}
\medskip
\vskip .05in

More specifically one can ask

\vskip .05in
\medskip

{\bf Question} 1. Is the question of whether or not an elliptic element of $PSL(2,{\mathbb{C}})$ is of finite order  decidable in some modified bit model?
\medskip

{\bf Question} 2. Is the Problem: determine whether $\arccos x$ is a rational multiple of $\pi$ decidable?

\medskip


{\sl Decidable} $\implies$ halting set is semi-algebraic.

\medskip

Kapovich's theorem, that $PSL(2,\mathbb{C})$ is not decidable in the BSS model depends upon the fact that the Maskit slice is not because cusps which correspond to trace $2$ are dense in the boundary. The Maskit slice in the character variety of the punctured torus is the set where  $w \in  \mathbb{C}$ which has a discrete image for  the group  generated by the  matrices
\vskip .05in
\centerline{ $\left(\begin{array}{cc}
    $w$ & $1$ \\
    $1$   & $0$ \\
  \end{array}\right)$,
  $\left(
  \begin{array}{cc}
    $1$ & $2$ \\
    $0$ & $1$ \\
  \end{array}
\right).$ }}

\medskip


{\bf Question:} 3. Are there interesting subsets of the finitely generated subgroups of $PSL(2,\mathbb{C})$ that are and are not decidable other than the Maskit slice?

\medskip
\vskip .05in

{\bf Question:} 4. Is there a way to incorporate $\T = \pm 2$ into a modified bit-computation model?



\section{Acknowledgement}

The author thanks the referees for multiple helpful comments.

\end{document}